\documentclass[12pt]{article}

\usepackage{color}

\usepackage{amssymb,amsthm,amsmath}
\usepackage{latexsym}
\usepackage{graphicx}
\usepackage{url}
\usepackage{fullpage}

\newtheorem{theorem}{Theorem}
\newtheorem{lemma}[theorem]{Lemma}

\newtheorem{prop}[theorem]{Proposition}

\newcommand{\N}{\mathcal{N}}

\newcommand\ex{\ensuremath{\mathrm{ex}}}

\begin{document}

\title{On supersaturation and stability for generalized Tur\'an problems}

\author{Anastasia Halfpap\thanks{Department of Mathematical Sciences,
University of Montana, Missoula, Montana 59812, USA.}
\and
Cory Palmer\footnotemark[1]
}

\maketitle

\begin{abstract}
Fix graphs $F$ and $H$. Let $\ex(n,H,F)$ denote the maximum number of copies of a graph $H$ in an $n$-vertex $F$-free graph. In this note we will give a new general supersaturation result for $\ex(n,H,F)$ in the case when $\chi(H) < \chi(F)$ as well as a new proof of a stability theorem for $\ex(n,K_r,F)$.
\end{abstract}

\section{Introduction}

Let $F$ be a graph. A graph $G$ is $F$-\emph{free} if it contains no copy of $F$ as a subgraph. Denote the maximum number of copies of a graph $H$ in an $n$-vertex $F$-free graph by
\[
\ex(n,H,F).
\]

After several sporadic results (a famous example is $\ex(n,C_5,K_3)$; see e.g., \cite{Gy,HaETAL,G2012}), the systematic study of the function $\ex(n,H,F)$ was initiated by Alon and Shikhelman \cite{alonsik}. An overview of results on $\ex(n,H,F)$ can be found in \cite{alonsik} and \cite{GePa}.

We denote the number of copies of a graph $H$ in a graph $G$ by $\N(H,G)$.
Recall that the \emph{Tur\'an graph} $T_{k-1}(n)$ is the $n$-vertex complete $(k-1)$-partite graph with classes of size as close as possible (i.e., classes differ by at most one vertex).
Zykov's ``symmetrization'' proof \cite{zykov} of Tur\'an's theorem gives the following generalization.

\begin{theorem}[Zykov, 1949]\label{zykov-theorem}
The Tur\'an graph $T_{k-1}(n)$ is the unique $n$-vertex $K_k$-free graph with the maximum number of copies of $K_r$. Thus,
\[
\ex(n,K_r,K_k) = \N(K_r, T_{k-1}(n)) \leq \binom{k-1}{r} \left \lceil \frac{n}{k-1}\right \rceil ^r.
\]
\end{theorem}

Zykov's theorem has been rediscovered and reproved several times (see e.g., \cite{alonsik, Bo, Er1}).

The Tur\'an graph $T_{k-1}$ contains no $k$-chromatic graph $F$, so we always have the trivial lower bound
\[
\N(H,T_{k-1}(n)) \leq \ex(n,H,F).
\]
Erd\H os-Stone-type generalizations of Theorem~\ref{zykov-theorem} were given in \cite{alonsik} and \cite{GePa}. We state them as a single theorem below.

\begin{theorem}[Alon-Shikhelman, 2016; Gerbner-Palmer, 2019]\label{ESS-gen}
Let $H$ be a graph and $F$ be a graph with chromatic number $k$, then
\[\ex(n,H,F) \leq \ex(n,H,K_k) + o(n^{|V(H)|}).\]
Thus, when $H=K_r$,
\[
 \ex(n, K_r, F) = \binom{k-1}{r}\left(\frac{n}{k-1}\right)^r + o(n^r).
\]
\end{theorem}

Note that the first part of Theorem~\ref{ESS-gen} only gives a useful upper-bound if $\ex(n,H,K_k) = \Omega(n^{|H|})$, which happens if and only if $K_k$ is not a subgraph of $H$. 

An important description of the \emph{degenerate case} is given by Alon and Shikhelman \cite{alonsik}. The \emph{blow-up} $G[t]$ of a graph $G$ is the graph resulting from replacing each vertex of $G$ with $t$ copies of itself.

\begin{prop}[Alon-Shikhelman, 2016]\label{degen-prop}
The function
    $\ex(n,H,F) = o(n^{|V(H)|})$ if and only if $F$ is a subgraph of a blow-up of $H$. Otherwise, $\ex(n,H,F) = \Omega(n^{|V(H)|})$.
\end{prop}

The purpose of this paper is to establish a general supersaturation result and give a new proof of a stability theorem for the function $\ex(n,H,F)$. Both of the main results are proved using modifications of standard proofs of stability and supersaturation for the ordinary Tur\'an function $\ex(n,F)$. Previous supersaturation results were given by  Cutler, Nir and Radcliffe \cite{cnr} who (among other things) proved a structural supersaturation for Theorem~\ref{zykov-theorem} (Zykov's Theorem) as well as a supersaturation result for the case when $H$ is a clique and $F$ is a star. Our first main result is a supersaturation theorem for $\ex(n,H,F)$ in the case when $\chi(H) < \chi(F)$.

\begin{theorem}[]\label{supersaturation}
Fix graphs $F, H$ on $f$ and $h$ vertices, respectively such that
$\chi(H) < \chi(F)$. 
 For $c > 0$, there exists $c_F > 0$ such that if $G$ is an $n$-vertex graph with 
\[
\N(H,G) > \ex(n,H,F) + c n^h,
\]
then $\N(F,G) \geq c_F n^f$.
\end{theorem}

Through a standard argument we can reprove Theorem~\ref{ESS-gen} via Theorem~\ref{supersaturation}. This and the proof of Theorem~\ref{supersaturation} will be given in Section~\ref{super-section}.

A general stability theorem for $\ex(n,K_r,F)$ was given by Ma and Qiu \cite{mq}. A second proof is due to Liu \cite{Liu}. We give give a new short proof of this theorem by making use of the standard stability theorem. 
This will be discussed in Section~\ref{stability-section}.

\begin{theorem}[Ma-Qiu, 2019]\label{stability}


Fix integers $r < k$ and let $F$ be graph with $\chi(F)=k$.
 If $G$ is an $n$-vertex $F$-free graph with 
 \[
 \N(K_r,G) > \ex(n,K_r,F) - o(n^r),
 \]
 then $G$ can be obtained from $T_{k-1}(n)$ by adding and removing $o(n^2)$ edges.
\end{theorem}

\section{Supersaturation}\label{super-section}
We begin with a modification of a result of Katona, Nemetz and Simonovits \cite{KNS}.

\begin{lemma}\label{mono-dec}
Let $H$ and $F$ be graphs. Then
 \[
\frac{\ex(n, H , F)}{\binom{n}{|V(H)|}}.
\]
is monotone decreasing as $n$ increases.
\end{lemma}

\begin{proof}
Suppose $G$ is an $n$-vertex $F$-free graph with the maximum number of copies of $H$. We double-count the pair $(H,v)$ where $H$ is a copy of the graph $H$ in $G$ and $v$ is a vertex not incident to $H$.
We can fix $H$ in $\ex(n,H,F)$ ways and then choose $v$ in $n-|V(H)|$ ways. On the other hand, there are $n$ ways to fix $v$ and on the remaining $n-1$ vertices there are at most
$\ex(n-1,H,F)$ copies of $H$. Thus,
\[
(n-|V(H)|)\cdot \ex(n,H,F) \leq n\cdot \ex(n-1,H,F).
\]
Solving for $\ex(n,H,F)$ and dividing both sides by $\binom{n}{|V(H)|}$ gives
\[
\frac{\ex(n,H,F)}{\binom{n}{|V(H)|}} \leq \frac{n}{n-|V(H)|} \frac{\ex(n-1,H,F)}{\binom{n}{|V(H)|}} = \frac{\ex(n-1,H,F)}{\binom{n-1}{|V(H)|}}.
\]
\end{proof}

Observe that if $F$ and $H$ satisfy $\chi(H) < \chi(F)$, then
\[
\frac{\ex(n, H , F)}{\binom{n}{|V(H)|}}
\]
is monotone decreasing by Lemma~\ref{mono-dec} and bounded below by the number of copies of $H$ in 
the Tur\'an graph $T_{\chi(F)-1}(n)$. This implies that
\[
\pi(H,F) =\lim_{n \rightarrow \infty} \frac{\ex(n, H , F)}{\binom{n}{|V(H)|}}
\]
exists. 

We are now ready to prove supersaturation in the generalized setting. The argument is essentially the same as an averaging argument used to prove supersaturation for hypergraphs (see \cite{super, Er4}).

\begin{proof}[Proof of Theorem~\ref{supersaturation}.]
Fix graphs $F, H$ on $f$ and $h$ vertices, respectively such that
$\chi(H) < \chi(F)$. 
Let 
\[
q=\pi(H,F)=\lim_{n \rightarrow \infty} \frac{\ex(n, H , F)}{\binom{n}{h}}.
\]
Fix $c > 0$ and suppose $G$ is an $n$-vertex graph with 
\[
\N(H,G) > \ex(n,H,F) + c n^h \geq (q +c) \binom{n}{h}.
\]
Choose $m$ such that 
\[
\ex(m,H,F) \leq \left(q + \frac{c}{2}\right)\binom{m}{h}.
\]
Assume (for the sake of a contradiction) that there are less than $\frac{c}{2 \cdot h!} \binom{n}{m}$ sets of $m$ vertices spanning more than $\left(q+\frac{c}{2}\right)\binom{m}{h}$
copies of $H$. Note that among $m$ vertices there are at most $\binom{m}{h}h!$ distinct copies of $H$. Therefore,
\begin{align*}
\sum_{S \in \binom{V(G)}{m}} \N(H,S) & <  \frac{c}{2 \cdot h!}\binom{n}{m}\binom{m}{h}h! + \binom{n}{m}\left(q + \frac{c}{2}\right)\binom{m}{h} \\
& =  \left( q + c \right)\binom{n}{m}\binom{m}{h}.
\end{align*}
On the other hand,
each copy of $H$ in $G$ is contained in $\binom{n - h}{m - h}$ vertex sets of size $m$, so
\[
\sum_{S \in \binom{V(G)}{m}} \N(H,S) = \binom{n - h}{m - h}\N(H,G) \geq  \binom{n - h}{m - h} \left(q + c\right) \binom{n}{h}=  \left(q + c\right) \binom{n}{m} \binom{m}{h}.
\]
Combining these two estimates for $\sum \N(H,S)$ gives a contradiction.
Therefore, there are at least $\frac{c}{2 \cdot h!} \binom{n}{m}$ sets of $m$ vertices spanning more than $\left(q+\frac{c}{2}\right)\binom{m}{h} \geq \ex(m,H,F)$
copies of $H$. Each of these $m$-sets contains a copy of $F$ and each copy of $F$ in $G$ is counted at most $\binom{n-f}{m - f}$ times in this way.
Therefore, the number of copies of $F$ in $G$ is
\[
\N(F,G) \geq \frac{c}{2 \cdot h!} \binom{n}{m}{\binom{n-f}{m - f}}^{-1} \geq c_F n^f
\]
for $c_F$ small enough.
\end{proof}

\begin{proof}[Proof of Theorem~\ref{ESS-gen}.]
 Let $H$ and $F$ be graphs on $h$ and $f$ vertices, respectively.
    Fix any $c>0$ and
    suppose $G$ is an $n$-vertex $F$-free graph with
    \[
    \N(G,H) > \ex(n,H,F)+ c n^h.
    \]
    Then Theorem~\ref{supersaturation} implies that
    $G$ contains at least $c_F n^f$ copies of $F$. Therefore, Proposition~\ref{degen-prop} implies that $G$ contains blow up $F[t]$ of $F$. 
    
    This implies that 
    \[
    \ex(n,H,F[t]) = \ex(n,H,F) +o(n^h).
    \]
    Now, as $F$ is contained in a blow-up $K_k[t]$ of $K_k$ for some $t$ large enough, we have
    \[
    \ex(n,H,F) \leq \ex(n,H,K_k[t]) \leq \ex(n,H,K_k) + o(n^h).
    \]
    
\end{proof}

\section{Stability}\label{stability-section}
We begin with two lemmas. The first will demonstrate that a $K_k$-free graph with nearly the extremal number of copies of $H$ contains a large subgraph in which every vertex is contained in many copies of $H$. The lemma is an argument of Norin \cite{norin} but adjusted to the subgraph counting context.

\begin{lemma}\label{norine-lemma}
Fix positive integers $k>r$. For $\alpha>0$, there exists $\beta > 0$ and $n_0 > 0$ such that every $K_k$-free graph $G$ with $|V(G)| \geq n_0$ and
\[
\N(H,G) > (1-\beta)\pi(H,K_k)\frac{|V(G)|^r}{r!}
\]
contains either

\noindent \textbf{1.} a subgraph $G'$ with $|V(G')| > (1- \alpha)|V(G)|$ such that every vertex of $G'$ is contained in more than
\[
(1-\alpha) \pi(H,K_k) \frac{|V(G')|^{r-1}}{(r-1)!}
\]
copies of $H$, or

 \noindent \textbf{2.} a subgraph $G'$ with $|V(G')| = \lfloor(1- \alpha)|V(G)|\rfloor$ and $\N(H,G') > \pi(H,K_k)\frac{|V(G')|^r}{r!}$. 
\end{lemma}

\begin{proof}
Choose $\delta$ so that $(1 - \delta)^2 > 1 - \frac{\alpha}{2}$ and $\delta < \frac{r \alpha^2}{2}$. Choose $n_0$ so that $n^r \geq (n -1 )^r + (1 - \delta)r n^{r-1}$ for all $n \geq (1 - \alpha)n_0$. If every vertex of $G$ belongs to more than $(1-\alpha) \pi(H,K_k) \frac{|V(G)|^{r-1}}{(r-1)!}$ copies of $H$, then we are done. If not, delete a vertex contained in the minimum number of copies of $H$ to obtain a subgraph $G_1$ of $G$. Repeat this procedure to obtain subgraphs $G_2, G_3, $ etc. If we reach a graph $G'$ that satisfies the lemma, then we are done. Therefore, suppose we have reached a graph $G_m$ such that $m = \lceil \alpha n\rceil$. We shall prove by induction on $\ell$ that
\begin{equation}\label{induct}
\N(H,G_\ell) > \Big( 1 - \frac{m - \ell}{m}\beta \Big)\pi(H,K_k)\frac{|V(G_\ell)|^r}{r!}.
\end{equation}
for $\ell \leq m$. 

The base case $\ell = 0$ follows from the hypotheses of the lemma.  So put $0 < \ell \leq m$ and assume (\ref{induct}) holds for $\ell-1$.  
For ease of notation put $n' = |V(G_{\ell-1})| = |V(G)|-\ell+1$.
Now (applying the induction hypothesis for $G_{\ell-1}$, and the choices of $\delta$ and $n_0$) we have
\begin{flalign*}
\frac{\N(H,G_\ell)}{\pi(H,K_k)} &\geq \frac{\N(H,G_{\ell-1})}{\pi(H,K_k)} - (1 -\alpha)\frac{(n')^{r-1}}{(r-1)!} \\
&\geq \Big(1 - \frac{m - \ell + 1}{m}\beta \Big)\frac{(n')^r}{r!} - (1 - \alpha)\frac{(n')^{r-1}}{(r-1)!} \\
&\geq \Big(1 - \frac{m - \ell + 1}{m} \beta \Big)\Big( \frac{(n' - 1)^r}{r!} + (1- \beta)\frac{(n')^{r-1}}{(r-1)!} \Big)- (1 - \alpha)\frac{(n')^{r-1}}{(r-1)!} \\
&\geq \Big( 1 - \frac{m - \ell}{m}\beta \Big)\frac{(n' - 1)^r}{r!} - \frac{\beta}{\alpha n}\frac{(n' - 1)^r}{r!} + \frac{\alpha}{2}\frac{(n')^{r-1}}{(r-1)!} \\
&\geq \Big( 1 - \frac{m - \ell}{m}\beta \Big) \frac{(n' - 1)^r}{r!} - \Big(\frac{\alpha}{2} - \frac{\beta}{\alpha r}\Big)\frac{(n')^{r-1}}{(r-1)!} \\
&> \Big( 1 - \frac{m - \ell}{m}\beta \Big) \frac{(n' - 1)^r}{r!}.
\end{flalign*}
Multiplying through by $\pi(H,K_k)$ proves (\ref{induct}).
When $m=\ell$ the inequality (\ref{induct}) gives
\[
\N(H,G_m) > \pi(H,K_k)\frac{|V(G_m)|^r}{r!}
\]
which completes the proof of the lemma.
\end{proof}

Our second lemma gives a lower bound on vertex degrees in $K_k$-free graphs with many copies of $K_r$.

\begin{lemma}\label{degree-lemma}
Let $G$ be an $n$-vertex, $K_k$-free graph and $x \in V(G)$. If 
\begin{equation}\label{kr-deg}
  \N(K_{r-1}, N(x)) \geq (1- \alpha) r \binom{k-1}{r} \Big( \frac{1}{k-1}\Big)^r n^{r-1},  
\end{equation}
then
\[
d(x) \geq (1-\alpha)^{1/(r-1)} \frac{k-2}{k-1}n - (k-3).
\]
\end{lemma}

\begin{proof}
The neighborhood $N(x)$ is $K_{k-1}$-free as $G$ is $K_k$-free. Therefore, by Theorem~\ref{zykov-theorem} we have
\begin{align*}
\N(K_{r-1}, N(x)) \leq \ex(|N(x)|, K_{r-1}, K_{k-1}) & \leq \binom{k-2}{r-1} \Big\lceil \frac{d(x)}{k-2}  \Big\rceil ^{r-1} \\
& \leq \binom{k-2}{r-1} \Big( \frac{d(x) + (k-3)}{k-2} \Big)^{r-1}.
\end{align*}
Combining the above estimate for $\N(K_{r-1}, N(x))$ with (\ref{kr-deg}) and solving for $d(x)$ completes the proof.
\end{proof}

Lemma~\ref{degree-lemma} implies that if each vertex of $G$
is contained in at least $\ex(n,K_r,K_k) \frac{r}{n} - o(n^{r-1})$ copies of $K_r$, then $e(G) \geq \left(1-\frac{1}{k-1}\right)\frac{n^2}{2} - o(n^2)$.

We will need a standard stability result for edges (see \cite{sim-stab}).

\begin{theorem}[Stability theorem]\label{edge-stability}
    Let $G$ be an $n$-vertex $F$-free graph with
    \[ 
    \left(1- \frac{1}{\chi(F)-1} \right) \frac{n^2}{2} - o(n^2).
    \]
    Then $G$ can be obtained from $T_{\chi(F)-1}(n)$ by adding and removing $o(n^2)$ edges.
\end{theorem}

\begin{proof}[Proof of Theorem~\ref{stability}.]
Fix integers $r<k$ and let $F$ be a graph with $\chi(F)= k$.
Let $G$ be an $n$-vertex $F$-free graph with
\[
\N(K_r,G) > \ex(n,K_r,F) - o(n^r).
\]
As $G$ is $F$-free with $\chi(F)=k$, a removal lemma due to Erd\H os Frankl and R\"odl \cite{EFR} asserts that $G$ can be made $K_k$-free with the removal of 
$o(n^2)$ edges\footnote{Note that when $F=K_k$ we may skip the use of the removal lemma.}. Removing $o(n^2)$ edges destroys at most $o(n^2) \cdot n^{r-1} = o(n^r)$ copies of $K_r$. Let $G'$ be the resulting $K_k$-free subgraph of $G$.
Now
\[
\N(K_r,G') > \ex(n,K_r,F) - o(n^r) \geq \ex(n,K_r,K_k) - o(n^r). 
\]
Let us apply Lemma~\ref{norine-lemma} to $G'$. Observe that the second outcome of Lemma~\ref{norine-lemma} is impossible here as it would imply that $G'$ contains a subgraph $K_k$. Therefore, the first outcome gives that $G'$ contains a subgraph $G''$
on $n'' \geq (1-\alpha)n$ vertices such that
each vertex of $G''$ is contained in $\ex(n,K_r,K_k)\frac{r}{n} - o(n^{r-1})$ copies of $K_r$. Applying Lemma~\ref{degree-lemma} to $G''$ gives that every degree in 
$G''$ is at least $\left(1-\frac{1}{k-1}\right)n - o(n)$ and therefore
\[
e(G'') \geq \left(1-\frac{1}{k-1}\right) \frac{n^2}{2} - o(n^2). 
\]
As $e(G) - e(G'') = o(n^2)$ we may apply the Stability theorem (Theorem~\ref{edge-stability}) to complete the proof.
\end{proof}

\end{document}